\documentclass[11pt, a4paper]{amsart}

\usepackage{amsmath, amsthm, amssymb}

\newcommand{\bC}{\mathbb{C}}
\newcommand{\sP}{\mathsf{P}}
\newcommand{\bP}{\mathbb{P}}
\newcommand{\cS}{\mathcal{S}}
\newcommand{\can}{\operatorname{can}}
\newcommand{\sd}{\mathsf{d}}
\newcommand{\rd}{\mathrm{d}}
\newcommand{\bR}{\mathbb{R}}
\newcommand{\cZ}{\mathcal{Z}}
\newcommand{\ord}{\operatorname{ord}}
\newcommand{\bN}{\mathbb{N}}
\newcommand{\diam}{\operatorname{diam}}
\newcommand{\PGL}{\mathit{PGL}}
\newcommand{\PSU}{\mathit{PSU}}
\newcommand{\cO}{\mathcal{O}}
\newcommand{\bD}{\mathbb{D}}
\newcommand{\sH}{\mathsf{H}}
\newcommand{\diag}{\operatorname{diag}}
\newcommand{\supp}{\operatorname{supp}}
\newcommand{\Id}{\mathrm{Id}}

\theoremstyle{plain}
\newtheorem{theorem}{Theorem}[section]
\newtheorem{lemma}[theorem]{Lemma}

\newtheorem{mainth}{Theorem}

\theoremstyle{definition}

\newtheorem{notation}[theorem]{Notation}

\newtheorem*{acknowledgement}{Acknowledgement}
\theoremstyle{remark}
\newtheorem{remark}[theorem]{Remark}
\newtheorem{fact}[theorem]{Fact}

\numberwithin{equation}{section}

\begin{document} 

\title[A Mahler-type estimate on the Berkovich projective line]{
A Mahler-type estimate of weighted Fekete sums on the Berkovich projective line}

\author[Y\^usuke Okuyama]{Y\^usuke Okuyama}
\address{
Division of Mathematics,
Kyoto Institute of Technology,
Sakyo-ku, Kyoto 606-8585 Japan.}
\email{okuyama@kit.ac.jp}

\date{\today}




\begin{abstract}
We establish a Mahler-type estimate
of weighted Fekete sums on the Berkovich projective line over
an algebraically closed field of possibly positive characteristic
that is complete with respect to
a non-trivial and possibly non-archimedean absolute value.
\end{abstract}

\maketitle 

\section{Introduction}\label{sec:intro}

Let $K$ be an algebraically closed field of possibly positive characteristic
that is complete with respect to
a non-trivial and possibly non-archimedean absolute value $|\cdot|$. 
Recall that $K$ is said to be {\itshape non-archimedean} 
if the {\itshape strong triangle inequality}
\begin{gather*}
 |z+w|\le\max\{|z|,|w|\} 
\end{gather*}
holds for every $z,w\in K$, and otherwise,
to be {\itshape archimedean}.
It is known that $K\cong\bC$ if and only if $K$ is archimedean.
The {\itshape Berkovich} projective line $\sP^1=\sP^1(K)$ is
a compact augmentation of the (classical) projective line 
$\bP^1=\bP^1(K)$.
It is known that $\sP^1\cong\bP^1$ if and only if $K$ is archimedean. 

For archimedean $K\cong\bC$,
the ($\log$ of the) classical {\itshape Mahler's estimate} of the Fekete 
{\itshape product} $\prod_{i=1}^N\prod_{j:\,j\neq i}|z_i-z_j|$ 
for any $N$ distinct points $z_1,\ldots,z_N$ in $K$ is
\begin{gather}
 \sum_{i=1}^N\sum_{j:\,j\neq i}(\log|z_i-z_j|-\log\max\{1,|z_i|\}-\log\max\{1,|z_j|\})
 \le N\log N\label{eq:mahlerorig}
\end{gather}
(\cite[Theorem 1]{Mahler64}). 
Aiming to study the field of definition for periodic points of 
a rational function over a number field, 
Benedetto \cite[Lemma 4.1]{Benedetto07} (in the case of $f\in K[z]$)
and Baker \cite[Theorem 1.1]{Baker06} (in the case of $f\in K(z)$)
generalized \eqref{eq:mahlerorig} as follows;
for every $f\in K(z)$ of degree $>1$,
\begin{gather}
 \sum_{i=1}^N\sum_{j:\,j\neq i}(\log[z_i,z_j]-g_f(z_i)-g_f(z_j))\le C\cdot N\log N\label{eq:BB}
\end{gather}
for any $N$ distinct points $z_1,\ldots,z_N$ in $\bP^1$,
where $C\ge 0$ is an effective constant independent of $z_1,\ldots,z_N$.
Here, $[z,w]$ is the normalized chordal distance on $\bP^1$ 
(see Notation \ref{th:notation} for the definition) and
$g_f$ is the dynamical Green function of $f$ on $\sP^1$ 
(see Fact \ref{th:dynamics} for the definition of $g_f$). 

The proofs of \eqref{eq:mahlerorig} and \eqref{eq:BB} were based on 
Hadamard's inequality, taking into account the geometry of the (homogeneous)
filled Julia set of (the non-degenerate homogeneous lift of) $f$.
One of our aims in this article is to give a simple proof of \eqref{eq:BB}
with, in general, the asymptotically {\itshape best possible} lower estimate of 
the constant $C>0$. Our proof is
based on a simple formula on {\itshape weighted Fekete sums} and
some upper and lower estimates of {\itshape regularized Fekete sums}
(see \S \ref{sec:Fekete} and \S \ref{th:reguFekete}, respectively)
from a potential theory on $\sP^1$.
For the foundation of the potential theory on $\sP^1$
for non-archimedean $K$, see Baker--Rumely \cite{BR10}, 
Favre--Rivera-Letelier \cite{FR09}, Thuillier \cite{ThuillierThesis}, 
and also Jonsson \cite{Jonsson15}. 
In the following, we adopt the notation in \cite{OkuDivisor}. 


\begin{notation}[Potential theory on $\sP^1$]\label{th:notation}
 Let $\pi:K^2\setminus\{(0,0)\}\to\bP^1=\bP^1(K)$ be the canonical projection
so that $\pi(p_0,p_1)=p_1/p_0\in K$ if $p_0\neq 0$ and that $\pi(0,1)=\infty$.
On $K^2$, let $\|(p_0,p_1)\|$ be 
the maximal norm $\max\{|p_0|,|p_1|\}$ (for non-archimedean $K$) 
or the Euclidean norm $\sqrt{|p_0|^2+|p_1|^2}$ (for archimedean $K$).
With the wedge product $(z_0,z_1)\wedge(w_0,w_1):=z_0w_1-z_1w_0$ 
on $K^2$, 
the {\itshape normalized chordal metric} $[z,w]$ on $\bP^1$ is 
the function
\begin{gather*}
 (z,w)\mapsto [z,w]:=|p\wedge q|/(\|p\|\cdot\|q\|)
\end{gather*}
on $\bP^1\times\bP^1$, where $p\in\pi^{-1}(z),q\in\pi^{-1}(w)$.
For non-archimedean $K$,
the {\itshape generalized Hsia kernel} 
$[\cS,\cS']_{\can}$ on $\sP^1$ {\itshape with respect to} $\cS_{\can}$
is the unique (jointly) upper semicontinuous 
and separately continuous extension to $\sP^1\times\sP^1$ 
of the chordal distance function $(z,w)\mapsto [z,w]$ on $\bP^1\times\bP^1$ 
(see \S \ref{th:kernels} for the definition of $[\cS,\cS']_{\can}$).
For archimedean $K$, by convention,
the kernel function $[z,w]_{\can}$ on $\sP^1\cong\bP^1$
is defined by $[z,w]$ itself. 

Let $\delta_{\cS}$ be the Dirac measure on $\sP^1$
at a point $\cS\in\sP^1$. The probability Radon measure $\Omega_{\can}$
on $\sP^1$ is defined as
\begin{gather*}
 \Omega_{\can}:=\begin{cases}
		\delta_{\cS_{\can}} & \text{if $K$ is non-archimedean},\\
		\omega & \text{if $K$ is archimedean},
	       \end{cases}
\end{gather*}
where $\cS_{\can}$ is
the {\itshape canonical $($or Gauss$)$ point}
in $\sP^1$ for non-archimedean $K$ (see \S \ref{th:berkovich} for
the definition), 
and $\omega$ is the Fubini-Study area element on $\bP^1$
normalized as $\omega(\bP^1)=1$ for archimedean $K$. 
The Laplacian $\Delta$ on $\sP^1$ is normalized so that for each $\cS'\in\sP^1$, 
\begin{gather*}
 \Delta\log[\cdot,\cS']_{\can}=\delta_{\cS'}-\Omega_{\can}
\end{gather*}
on $\sP^1$
(for non-archimedean $K$, see \cite[\S5.4]{BR10}, \cite[\S2.4]{FR09}; 
in \cite{BR10} the opposite sign convention on $\Delta$ is adopted).

A {\itshape continuous weight $g$ on $\sP^1$} is
a continuous function on $\sP^1$ such that
$\mu^g:=\Delta g+\Omega_{\can}$ 
is a probability Radon measure on $\sP^1$.
For a continuous weight $g$ on $\sP^1$,
the $g$-{\itshape potential kernel} on $\sP^1$
(or the {\itshape negative of} an Arakelov Green kernel function on $\sP^1$
relative to $\mu^g$ \cite[\S 8.10]{BR10}) is a (jointly) 
upper semicontinuous function
\begin{gather}
 \Phi_g(\cS,\cS'):=\log[\cS,\cS']_{\can}-g(\cS)-g(\cS')\label{eq:kernel} 
\end{gather}
on $\sP^1\times\sP^1$, and the {\itshape $g$-equilibrium energy}
$V_g$ {\itshape of $\sP^1$} (in fact $V_g\in\bR$)
is the supremum of the $g$-energy functional
$\nu\mapsto\int_{\sP^1\times\sP^1}\Phi_g\rd(\nu\times\nu)\in[-\infty,+\infty)$
over all probability Radon measures $\nu$ on $\sP^1$. 
A probability Radon measure $\nu$ on $\sP^1$ at which
the above $g$-energy functional attains the supremum $V_g$
is called a {\itshape $g$-equilibrium mass distribution on} $\sP^1$;
in fact, $\mu^g$ is the unique $g$-equilibrium mass distribution on $\sP^1$
(for non-archimedean $K$, see \cite[Theorem 8.67, Proposition 8.70]{BR10}).

A {\itshape normalized weight $g$ on} $\sP^1$ 
is a continuous weight on $\sP^1$ satisfying $V_g=0$
(for every continuous weight $g$ on $\sP^1$, $\overline{g}:=g+V_g/2$ is
the unique normalized weight on $\sP^1$ such that $\mu^{\overline{g}}=\mu^g$).
\end{notation}

One of our principal results is the following Mahler-type estimate.

\begin{mainth}\label{th:Holder}
Let $K$ be an algebraically closed field of possibly positive characteristic
that is complete with respect to
a non-trivial and possibly non-archimedean absolute value.
Let $g$ be a normalized weight on $\sP^1$, and 
suppose that the restriction $g|\bP^1$ is a $1/\kappa$-H\"older continuous function
on $(\bP^1,[z,w])$ for some $\kappa\ge 1$.
Then setting $C:=\sup_{z,w\in\bP^1:\,\text{distinct}}|g(z)-g(w)|/[z,w]^{1/\kappa}\in\bR_{\ge 0}$, for every non-empty finite subset $F$ in $\bP^1$, we have
\begin{multline}
\sum_{z\in F}\sum_{w\in F\setminus\{z\}}\Phi_g(z,w)\\
 \le\kappa\cdot(\#F)\log(\#F)
 +2(\#F)\Bigl(C'+\epsilon_K\cdot(\#F)^{1-\kappa}+\sup_{\sP^1}|g|\Bigr),
 \tag{$\ref{eq:mahler}'$}
\label{eq:maherHolder}
\end{multline}
where we also set $C':=C\cdot 2^{1/\kappa}$ if $K$ is non-archimedean, and  
$C':=C$ otherwise, 
and $\epsilon_K:=1$ if $K$ is archimedean, and $\epsilon_K:=0$
otherwise.

In particular,
\begin{gather}
 \limsup_{N\to\infty}\biggl(\sup_{F\subset\bP^1:\,0<\#F\le N}
 \frac{\sum_{z\in F}\sum_{w\in F\setminus\{z\}}\Phi_g(z,w)}{(\#F)\log(\#F)}\biggr)\le\kappa.\label{eq:asymp}
\end{gather} 
\end{mainth}

Theorem \ref{th:Holder} is a consequence of \eqref{eq:mahler}
in Theorem \ref{th:mahler} stated and shown in Section \ref{sec:mahler}, 
which is a little technical but
applies to any normalized weight $g$ on $\sP^1$, involving
the {\itshape restricted} modulus of continuity
\begin{gather}
 \eta_{g,F}(\epsilon)
 :=\max_{z\in F}\biggl(\sup_{\cS\in\sP^1:\, \sd(z,\cS)\le\epsilon}|g(z)-g(\cS)|\biggr)\quad\text{on }[0,1]\label{eq:moduluspre}
\end{gather}
of $g$ {\itshape around} a non-empty finite subset $F$ in $\bP^1$
with respect to the metric
$\sd$ on $\sP^1$ (see \S \ref{th:kernels} for the definition of $\sd$).
%
The estimate \eqref{eq:BB} is obtained as a special case of Theorem \ref{th:Holder}
by recalling the following.

\begin{fact}\label{th:dynamics}
 For every $f\in K(z)$ of degree $d>1$,
 whose action on $\bP^1$ canonically extends to that on $\sP^1$,
 there is the weak limit 
$\mu_f=\lim_{n\to\infty}(f^n)^*\Omega_{\can}/d^n$ 
on $\sP^1$, which is called  
 the $f$-{\itshape equilibrium $(${\rm or} canonical$)$} measure
 on $\sP^1$ (for non-archimedean $K$,
 see \cite[\S10]{BR10}, \cite[\S2]{ChambertLoir06}, \cite[\S3.1]{FR09}).
 We call the unique normalized weight $g$ on $\sP^1$ such that
 $\mu^g=\mu_f$ on $\sP^1$
 the $f$-{\itshape dynamical Green function} on $\sP^1$ and denote it by $g_f$.
 It is known that $f:(\bP^1,[z,w])\to(\bP^1,[z,w])$ is Lipschitz continuous
 (for non-archimedean $K$, see \cite[Theorem 2]{KS09}) and that
 if for every $n\in\bN$, $f^n:(\bP^1,[z,w])\to(\bP^1,[z,w])$ 
 is $M_n$-Lipschitz continuous for some $M_n>d^n$,
 then  the restriction $g_f|\bP^1$ is
 $1/\kappa$-H\"older continuous on $(\bP^1,[z,w])$ 
 for every $\kappa>\limsup_{n\to\infty}(\log(M_n^{1/n}))/\log d$
 (see e.g.\ \cite[\S 6.6]{FR06}). 
\end{fact}

\subsection*{Organization of the article} In Section \ref{sec:background},
we recall background from potential theory on $\sP^1$ including
a few preparatory lemmas and facts, and
some details on the regularization of Dirac measures supported in $\bP^1$.
In Section \ref{sec:mahler}, we state and show Theorem \ref{th:mahler},
and then deduce Theorem \ref{th:Holder} from Theorem \ref{th:mahler}.
In Section \ref{sec:lower},
we include a deduction of the lower estimate \eqref{eq:lower}
of regularized Fekete sums, which plays a key role in the proof of
Theorem \ref{th:mahler}. In Section \ref{eq:Lipschitz}, we include
a few examples, for which \eqref{eq:asymp} is optimal.

\section{Background from potential theory on $\sP^1$}\label{sec:background}

Let $K$ be an algebraically closed field 
that is complete with respect to
a non-trivial 
absolute value $|\cdot|$. 

For more details including references of this section, see \cite{OkuDivisor}.

\subsection{The Berkovich projective line $\sP^1$ for non-archimedean $K$}\label{th:berkovich}
Suppose that $K$ is non-archimedean in this subsection.
A subset $B$ in $K$ is a ($K$-closed) {\itshape disk in} $K$
if this is written as
$\{z\in K:|z-a|\le r\}$ for some $a\in K$ and some $r\ge 0$.
By the strong triangle inequality,
{\itshape two disks in $K$ either nest or are disjoint}.
This alternative extends to any two decreasing infinite sequences 
of disks in $K$ so that they either {\itshape infinitely nest}
or {\itshape are eventually disjoint}, and induces 
the so called {\itshape cofinal} equivalence relation among them.
As a set, the set of all cofinal equivalence classes $\cS$ of
decreasing infinite sequences $(B_n)$ of disks in $K$ and in addition $\infty\in\bP^1$ is nothing but $\sP^1$ (\cite[p.\ 17]{Berkovichbook}); 
if $\cS\neq\infty$, then $B_{\cS}:=\bigcap_n B_n$ is independent of
the choice of $(B_n)$ and is itself a disk in $K$ unless $B_{\cS}=\emptyset$.
For examples,
the {\itshape canonical $($or Gauss$)$ point} $\cS_{\can}$ in $\sP^1$ is
 the cofinal equivalence class of the constant sequence $(B_n)$ 
 of disks $B_n\equiv\mathcal{O}_K$ in $K$, 
 where $\mathcal{O}_K:=\{z\in K:|z|\le 1\}$ is the ring of $K$-integers.
 Each $z\in\bP^1$ is identified with
 the cofinal equivalence class of the constant sequence $(B_n)$ 
 of disks $B_n\equiv\{z\}$ in $K$.

The above alternative among decreasing infinite sequences of disks in $K$ 
also induces a partial ordering $\succeq$ on $\sP^1$ so that for every
$\cS,\cS'\in\sP^1\setminus\{\infty\}$ satisfying 
$B_{\cS},B_{\cS'}\neq\emptyset$,
$\cS\succeq\cS'$ if and only if $B_{\cS}\supset B_{\cS'}$
(the description of $\succeq$ in the case that $B_{\cS}$ or
$B_{\cS'}$ is empty is a little complicate) and that $\infty\succeq\cS$
for every $\cS\in\sP^1$, 
and endows $\sP^1$
with the canonical {\itshape tree} structure
in the sense of Jonsson \cite[\S 2, Definition 2.2]{Jonsson15}.

The topology of $\sP^1$ coincides with the weak topology
induced by the tree structure of $\sP^1$, and
$\sP^1$ is uniquely arcwise-connected and contains both
$\bP^1$ and $\sH^1=\sH^1(K):=\sP^1\setminus\bP^1$ as dense subsets.

\subsection{The kernels $[\cS,\cS']_{\can}$ and $|\cS-\cS'|_\infty$ and the distance $\sd$ on $\sP^1$}\label{th:kernels}
Suppose first that $K$ is non-archimedean. 
Let $\diam B$ be the diameter of a disk $B$ in $K$ with respect to $|\cdot|$.
For every $\cS\in\sP^1\setminus\{\infty\}$ represented by 
a decreasing infinite sequence $(B_n)$ of disks in $K$, set 
\begin{gather*}
 \diam\cS
 :=\lim_{n\to\infty}\diam B_n(=\diam B_{\cS}\text{ unless }B_{\cS}=\emptyset),
\end{gather*}
which is independent of the choice of $(B_n)$, and set $\diam\infty=+\infty$.
For every $\cS,\cS'\in\sP^1$, let $\cS\wedge\cS'$ be the smallest
$\cS''\in\sP^1$ satisfying $\cS''\succeq\cS$ and $\cS''\succeq\cS'$.
Under the convention that $\infty/(\infty^2)=0$,
the generalized Hsia kernel $[\cS,\cS']_{\can}$ on $\sP^1$
with respect to $\cS_{\can}$ is defined by the function
\begin{gather}
[\cS,\cS']_{\can}
:=\frac{\diam\cS''}{(\diam(\cS_{\can}\wedge\cS''))^2}\label{eq:Gromov}
\end{gather}
on $\sP^1\times\sP^1$,
where $\cS''$ is the unique point in $\sP^1$ lying between $\cS$ and $\cS'$,
between $\cS'$ and $\cS_{\can}$, and between $\cS_{\can}$ and $\cS$
with respect to $\succeq$ (see \cite[\S 3.4]{FR06}, \cite[\S 4.4]{BR10}).
Then, as mentioned in Section \ref{sec:intro},
the kernel function $(\cS,\cS')\mapsto[\cS,\cS']_{\can}$ on $\sP^1\times\sP^1$
is the unique (jointly)
upper semicontinuous and separately continuous extension 
to $\sP^1\times\sP^1$ of the chordal distance function 
$(z,w)\mapsto[z,w]$ on $\bP^1\times\bP^1$.
If $K$ is archimedean, then $[z,w]_{\can}$ 
is defined by $[z,w]$ itself on $\sP^1\cong\bP^1$, by convention.  

No matter whether $K$ is archimedean or non-archimedean,
set the function
\begin{gather}
 \sd(\cS,\cS'):=[\cS,\cS']_{\can}
-\frac{[\cS,\cS]_{\can}+[\cS',\cS']_{\can}}{2}\label{eq:small}
\end{gather}
on $\sP^1\times\sP^1$.
If $K$ is archimedean, then the function $\sd(z,w)$ on $\sP^1\times\sP^1$
is nothing but
the chordal metric $[z,w]$ on $\sP^1\cong\bP^1$. If $K$ is 
non-archimedean, the function $\sd$ extends $[z,w]$ to $\sP^1$ 
as a metric on $\sP^1$ (see \cite[\S 4.7]{FR06}, \cite[\S 2.7]{BR10}),
and is called the {\itshape small model metric} on $\sP^1$.

Although the difference $\cS-\cS'$ between $\cS,\cS'\in\sP^1$
is defined only if both $\cS,\cS'$ are in $K$, set the function 
\begin{gather}
 |\cS-\cS'|_{\infty}
:=\frac{[\cS,\cS']_{\can}}{[\cS,\infty]_{\can}\cdot[\cS',\infty]_{\can}}\label{eq:Hsia}
\end{gather}
on $\sP^1\times\sP^1$,
under the convention $0/(0^2)=\infty$. If $K$ is archimedean, then
the restriction $|z-w|_\infty$ to $K\times K$
is nothing but the euclidean metric $|z-w|$ on $K\cong\bC$.
If $K$ is non-archimedean, $|\cS-\cS'|_{\infty}$
is the unique (jointly)
upper semicontinuous and separately continuous extension 
to $\sP^1\times\sP^1$ of 
the distance function
$(z,w)\mapsto|z-w|=[z,w]/([z,\infty]\cdot[w,\infty])$ on $K\times K$, 
and is called
the (original) {\itshape Hsia kernel} on $\sP^1$ (see \cite[\S 4.4]{BR10}).


\begin{lemma}\label{th:comparison}
On $\bP^1\times\sP^1$, $\sd(z,\cS)\ge[z,\cS]_{\can}/2$. 
\end{lemma} 

\begin{proof}
There is nothing to show for archimedean $K$, so suppose that
$K$ is non-archimedean. For every $z\in\bP^1$ and every $\cS\in\sP^1$,
 by the definition \eqref{eq:small} of the metric $\sd$,
 we have $\sd(z,\cS)=[z,\cS]_{\can}-[\cS,\cS]_{\can}/2$, and
 by the definition \eqref{eq:Gromov} of the kernel $[\cS,\cS']_{\can}$, 
 we have
\begin{align*}
 [z,\cS]_{\can}=&\begin{cases}
		 \diam(z\wedge\cS) & 
\text{if }\cS_{\can}\succeq\cS\text{ and }\cS_{\can}\succeq z,\\
		 \diam\cS_{\can} &
\text{if }\cS_{\can}\succeq\cS\text{ and }\cS_{\can}\not\succeq z,\\
		 1/\diam(\cS_{\can}\wedge z) &
\text{if }\cS_{\can}\not\succeq\cS\text{ and }\cS_{\can}\wedge\cS\succeq\cS_{\can}\wedge z,\\
		 1/\diam(\cS_{\can}\wedge\cS) &
\text{if }\cS_{\can}\not\succeq\cS\text{ and }\cS_{\can}\wedge z\succeq\cS_{\can}\wedge\cS
		\end{cases}\\
\ge&
 \begin{cases}
  \diam\cS & \text{if }\cS_{\can}\succeq\cS,\\
  (\diam(\cS_{\can}\wedge\cS))/(\diam(\cS_{\can}\wedge\cS))^2 & 
\text{if }\cS_{\can}\not\succeq\cS
 \end{cases}\\
\ge&
\begin{cases}
 \diam\cS & \text{if }\cS_{\can}\succeq\cS,\\
 (\diam\cS)/(\diam(\cS_{\can}\wedge\cS))^2 & \text{if }\cS_{\can}\not\succeq\cS
\end{cases}\\
=&[\cS,\cS]_{\can},
\end{align*}
which completes the proof.
\end{proof}

\subsection{The isometry group $U_K$ on $(\sP^1,\sd)$}\label{th:isometry}
The action on $\bP^1$ of a linear fractional transformation $h\in\PGL(2,K)$ 
uniquely extends to $\sP^1$ as a continuous automorphism 
on $\sP^1$, and induces the pullback $h^*$
and the push-forward $h_*=(h^{-1})^*$
on the space of all continuous functions on $\sP^1$
and, by duality, the space of all probability 
Radon measures on $\sP^1$
(see e.g.\ \cite[\S 2.3]{BR10}). 
Let $U_K$ be either the subgroup $\PSU(2,K)$ (for archimedean $K\cong\bC$) or  
the subgroup $\PGL(2,\cO_K)$ (for non-archimedean $K$) in $\PGL(2,K)$.
Each $h\in U_K$ acts on $(\bP^1,[z,w])$ isometrically (for non-archimedean $K$,
see e.g.\ \cite[\S 1]{Benedetto03}). Hence each $h\in U_K$ not only satisfies 
$[h(\cS),h(\cS')]_{\can}=[\cS,\cS']_{\can}$ on $\sP^1\times\sP^1$
by the separate continuity of both sides on $\sP^1\times\sP^1$ and 
the density of $\bP^1$ in $\sP^1$, but also acts on $(\sP^1,\sd)$ isometrically
(recall the definition \eqref{eq:small} of $\sd$).
Moreover, for every $h\in U_K$, we have
$h^*\Omega_{\can}=\Omega_{\can}$
on $\sP^1$; indeed, fixing $\cS\in\sP^1$, we have
\begin{multline*}
 h^*\Omega_{\can}
 =h^*\delta_{\cS}-\Delta h^*\log[\cdot,\cS]_{\can}
 =h^*\delta_{\cS}-\Delta \log[h(\cdot),h(h^{-1}(\cS))]_{\can}\\
 =\delta_{h^{-1}(\cS)}-\Delta \log[\cdot,h^{-1}(\cS)]_{\can}
 =\delta_{h^{-1}(\cS)}-(\delta_{h^{-1}(\cS)}-\Omega_{\can})
 =\Omega_{\can}\quad\text{on }\sP^1
\end{multline*}
(for the functoriality $h^*\Delta=\Delta h^*$
for non-archimedean $K$, see e.g.\ \cite[\S9]{BR10}).


\begin{lemma}\label{th:pullback}
For every normalized weight $g$ on $\sP^1$ and every $h\in U_K$,
$g\circ h$ is also a normalized weight on $\sP^1$. 
\end{lemma}

\begin{proof}
 We first compute as
 \begin{gather*}
 \mu^{g\circ h}=
 \Delta(g\circ h)+\Omega_{\can}
 =h^*(\Delta g)+\Omega_{\can}=h^*\mu^g-h^*\Omega_{\can}+\Omega_{\can}
 =h^*\mu^g,
 \end{gather*} 
 which is a probability Radon measure on $\sP^1$, so $g\circ h$ is
 a continuous weight on $\sP^1$. Next,
 by $[h(\cS),h(\cS')]_{\can}=[\cS,\cS']_{\can}$ on $\sP^1\times\sP^1$
 and the characterization of
 $\mu^g$ (resp.\ $\mu^{g\circ h}$) as the (unique) $g$-equilibrium 
 (resp.\ $g\circ h$-equilibrium) mass distribution on $\sP^1$,
 we have
 \begin{gather*}
 V_{g\circ h}=\int_{\sP^1\times\sP^1}\Phi_{g\circ h}\rd((h^*\mu^g)\times(h^*\mu^g))=\int_{\sP^1\times\sP^1}\Phi_g\rd(\mu^g\times\mu^g)=V_g=0,
 \end{gather*}
 which completes the proof.
\end{proof}

\subsection{A regularization of an effective divisor on $\sP^1$}
First, for every $z\in K$ and every $\epsilon>0$,
let us define the {\itshape $\epsilon$-regularization} of
the Dirac measure $\delta_z$ on $\sP^1$ by the probability
Radon measure
\begin{multline*}
 [z]_\epsilon:=\Delta\log\max\{\epsilon,|\cdot-z|_{\infty}\}+\delta_\infty\\
=\begin{cases}
	       \delta_{\pi_\epsilon(z)} & \text{for non-archimedean }K\\
	        m_{\{w\in K:|w-z|=\epsilon\}} & \text{for archimedean }K
	      \end{cases}
\end{multline*}
on $\sP^1$;
for non-archimedean $K$, the continuous mapping
$\pi_\epsilon:\bP^1\to\sH^1$ 
is defined so that for every $z\in K$,
$\pi_\epsilon(z)\in\sH^1$ is represented by
the constant sequence $(B_n)$ of disks $B_n\equiv\{w\in K:|w-z|\le\epsilon\}$
in $K$, and for archimedean $K$, $m_{\partial\bD(z,\epsilon)}$ 
is the $1$-dimensional Lebesgue measure
on the circle $\partial\bD(z,\epsilon)
:=\{w\in K:|w-z|=\epsilon\}$ in $K\cong\bC$ normalized as
$m_{\partial\bD(z,\epsilon)}(\partial\bD(z,\epsilon))=1$.
Next, with the involution $\iota(z):=1/z\in U_K$ on $\sP^1$, 
for every $\epsilon>0$, set
\begin{gather*}
 [\infty]_{\epsilon}:=\iota_*[0]_{\epsilon}.
\end{gather*}
Then for every $z\in\bP^1$ and every $\epsilon>0$, $[z]_\epsilon$ has
no atoms on $\bP^1$.
The following would justify the terminology
(cf.\ \cite[Lemme 4.8]{FR06},
\cite[\S 2.1]{FP15}).

\begin{lemma}\label{th:chordal}
For every $z\in\bP^1$ and every $\epsilon>0$,
the chordal potential 
$U^\#_{[z]_\epsilon}(\cdot):=\int_{\sP^1}\log[\cdot,\cS']_{\can}\rd[z]_\epsilon(\cS')$
of the $\epsilon$-regularization $[z]_{\epsilon}$ of $\delta_z$ on $\sP^1$
is a {\itshape continuous} function on $\sP^1$.
\end{lemma}

\begin{proof}
 Fix $z\in\bP^1$ and $\epsilon>0$.
 For non-archimedean $K$, if $z\in K$, then 
 $U^\#_{[z]_\epsilon}(\cdot)\equiv\log[\cdot,\pi_\epsilon(z)]_{\can}$ 
 is continuous on $\sP^1$
 by $\pi_\epsilon(z)\in\sH^1$ and the separate continuity of 
 $[\cS,\cS']_{\can}$ on $\sP^1\times\sP^1$. 
 For archimedean $K\cong\bC$, if $z\in K$, then
\begin{multline*}
 w\mapsto U^\#_{[z]_\epsilon}(w)
=\int_0^{2\pi}\log[w,z+\epsilon e^{i\theta}]\frac{\rd\theta}{2\pi}\\
 =\int_0^{2\pi}\log|(w-z)-\epsilon e^{i\theta}|\frac{\rd\theta}{2\pi}
 -\log\sqrt{1+|w|^2}
 -\int_0^{2\pi}\log\sqrt{1+|z+\epsilon e^{i\theta}|^2}\frac{\rd\theta}{2\pi}\\
\equiv\log(\max\{|w-z|,\epsilon\})/\sqrt{1+|w|^2})
 +U^\#_{[z]_\epsilon}(\infty)
\end{multline*}
is continuous on $\sP^1\cong\bP^1$.
Finally, no matter whether
$K$ is non-archimedean or archimedean, 
$U^\#_{[\infty]_\epsilon}\equiv U^\#_{[0]_\epsilon}\circ\iota$
is continuous on $\sP^1$ 
by the above continuity of $U^\#_{[0]_\epsilon}$ on $\sP^1$ and
the continuity of $\iota:\sP^1\to\sP^1$.
\end{proof}

\subsection{The $g$-Fekete sum with respect to an effective divisor on $\bP^1$}\label{sec:Fekete}
Every effective divisor $\cZ$ on $\bP^1=\bP^1(K)$
is regarded as a
positive and discrete Radon measure
\begin{gather*}
 \sum_{w\in\supp\cZ}(\ord_w\cZ)\cdot\delta_w
\end{gather*}
on $\sP^1$, which is denoted by the same $\cZ$.
Let
$\diag_{\bP^1}$ be the diagonal in $\bP^1\times\bP^1$.

For a continuous weight $g$ on $\sP^1$, the {\itshape $g$-Fekete sum}
with respect to an effective divisor $\cZ$ on $\bP^1$ is defined by
\begin{multline*}
 (\cZ,\cZ)_g:=\int_{(\sP^1\times\sP^1)\setminus\diag_{\bP^1}}
\Phi_g\rd(\cZ\times\cZ)\\
=\sum_{z\in\supp\cZ}\sum_{w\in\supp\cZ\setminus\{z\}}(\ord_z\cZ)(\ord_w\cZ)\cdot\Phi_g(z,w)\in\bR,
\end{multline*}
whose sign convention 
is opposite to ones of Favre--Rivera-Letelier's
Dirichlet forms in \cite{DujardinFavre08} and
is compatible with the $\log$ of the original
Fekete product (\cite{Fekete30, Fekete33}) in Section \ref{sec:intro}.

Every non-empty finite subset $F$ in $\bP^1$ is canonically regarded as 
the effective divisor $\cZ_F$ on $\bP^1$ such that
$\supp\cZ_F=F$ and that $\ord_w\cZ_F=1$ for every $w\in F$.
For a continuous weight $g$ on $\sP^1$ and
a non-empty finite subset $F$ in $\bP^1$, we also 
define the $g$-Fekete sum with respect to $F$ by
\begin{gather*}
 (F,F)_g:=(\cZ_F,\cZ_F)_g=\sum_{z\in F}\sum_{w\in F\setminus\{z\}}\Phi_g(z,w),
\end{gather*}
which satisfies a formula
\begin{gather}
  (F,F)_g=2\cdot\sum_{w\in F\setminus\{z\}}\Phi_g(z,w)
 +(F\setminus\{z\},F\setminus\{z\})_g\label{eq:formula}
\end{gather} 
for every $z\in F$. 
Recall Lemma \ref{th:pullback} here.

\begin{lemma}\label{th:pullbackprod}
 For every normalized weight $g$ on $\sP^1$,
 every effective divisor $\cZ$ on $\bP^1$, and every $h\in U_K$,
\begin{gather*}
  (\cZ,\cZ)_g=(h^*\cZ,h^*\cZ)_{g\circ h}.
\end{gather*}
\end{lemma}

\begin{proof}
 By $[h^{-1}(\cS),h^{-1}(\cS')]_{\can}=[\cS,\cS']_{\can}$ 
 on $\sP^1\times\sP^1$ for every $h\in U_K$, 
 we have
 \begin{multline*}
 (\cZ,\cZ)_g=
 \int_{(\sP^1\times\sP^1)\setminus\diag_{\bP^1}}
 \Phi_g\rd(\cZ\times\cZ)\\
 =\int_{(\sP^1\times\sP^1)\setminus\diag_{\bP^1}}
 (\log[h^{-1}(\cS),h^{-1}(\cS')]_{\can}-g(\cS)-g(\cS'))
 \rd(\cZ\times\cZ)(\cS,\cS')\\
 =\int_{(\sP^1\times\sP^1)\setminus\diag_{\bP^1}}
 \Phi_{g\circ h}\rd((h^*\cZ)\times (h^*\cZ))
 = (h^*\cZ,h^*\cZ)_{g\circ h}, 
 \end{multline*}
which completes the proof.
\end{proof}

\subsection{Estimates of regularized Fekete sums}\label{th:reguFekete}
For every $\epsilon>0$ and every effective divisor $\cZ$ on $\bP^1$,
the $\epsilon$-regularization of $\cZ$ is defined by
$\cZ_{\epsilon}:=\sum_{w\in\cZ}(\ord_w\cZ)\cdot[w]_\epsilon$
on $\sP^1$, and 
for every continuous weight $g$ on $\sP^1$,
the {\itshape $\epsilon$-regularized $g$-Fekete sum} with respect to $\cZ$ is
\begin{gather*}
 (\cZ_\epsilon,\cZ_\epsilon)_g
 :=\int_{\sP^1\times\sP^1}\Phi_g\rd(\cZ_\epsilon\times\cZ_\epsilon)
\in\bR.
\end{gather*}

\begin{fact}
 If the continuous weight $g$ is a normalized weight on $\sP^1$, then 
 for every $\epsilon>0$ and every effective divisor $\cZ$ on $\bP^1$, 
 the continuity of $U^\#_{[z]_\epsilon}(\cdot)$ on $\sP^1$ (Lemma \ref{th:chordal})
implies
 the {\itshape negativity}
 \begin{gather}
 (\cZ_\epsilon,\cZ_\epsilon)_g\le V_g=0\label{eq:negative}
 \end{gather}
 (cf.\ \cite[\S 2.5 et \S 4.5]{FR06}), and
 if in addition $\epsilon\in(0,1]$, then we also have
 \begin{multline}
 (\cZ_{\epsilon},\cZ_{\epsilon})_g
 \ge
 (\cZ,\cZ)_g
 +2\sum_{w\in\supp\cZ\setminus\{\infty\}}(\ord_w\cZ)^2\cdot\log[w,\infty]\\
 -2\sum_{w\in\supp\cZ}(\ord_w\cZ)^2g(w)\\
 +(\log\epsilon)\cdot(\cZ\times\cZ)(\diag_{\bP^1})
 -2(\deg\cZ)^2\cdot\hat{\eta}_{g,\supp\cZ}(\epsilon),\label{eq:lower}
 \end{multline}
 where for every non-empty subset $F$ in $\bP^1$,
 recalling the definition \eqref{eq:moduluspre} of the restricted 
 modulus of continuity $\eta_{g,F}:[0,1]\to\bR_{\ge 0}$ 
of $g$ around $F$ with respect to $\sd$, 
 the function $\hat{\eta}_{g,F}:[0,1]\to\bR_{\ge 0}$ is defined by 
\begin{gather}
 [0,1]\ni\epsilon\mapsto\hat{\eta}_{g,F}(\epsilon)
:=\eta_{g,F}(\epsilon)+
\begin{cases}
 0 & \text{for non-archimedean $K$},\\
 \epsilon & \text{for archimedean $K$}
\end{cases}\label{eq:modulus}
\end{gather}
(cf.\ \cite[\S 2.6 et \S 4.7]{FR06}).
\end{fact}

We will include a deduction of the latter \eqref{eq:lower}
in Section \ref{sec:lower}.

\section{Proofs of Theorems \ref{th:Holder} and \ref{th:mahler}}\label{sec:mahler}

Let $K$ be an algebraically closed field that is complete with respect to a
non-trivial absolute value, and
let $g$ be a normalized weight on $\sP^1$.
Recall the definition \eqref{eq:modulus} of $\hat{\eta}_{g,F}$
(and \eqref{eq:moduluspre} of $\eta_{g,F}$).
The following could be regarded as a refinement of
Favre--Rivera-Letelier \cite[Propositions 2.8 et 4.9]{FR06}.

\begin{mainth}\label{th:mahler}
Let $K$ be an algebraically closed field of possibly positive characteristic
that is complete with respect to
a non-trivial and possibly non-archimedean absolute value, and
let $g$ be a normalized weight on $\sP^1$. Then
for every non-empty finite subset $F$ in $\bP^1$ and every $\epsilon\in(0,1]$, 
\begin{gather}
(F,F)_g\le (\#F)\log(\epsilon^{-1})
+2(\#F)^2\cdot\hat{\eta}_{g,F}(\epsilon)+2(\#F)\sup_{\sP^1}|g|.\label{eq:mahler}
\end{gather}
\end{mainth}

\begin{remark}
By \cite[Propositions 2.8 et 4.9]{FR06} mentioned above,
we could assert that 
for every non-empty and finite
subset $F$ {\itshape in} $\bP^1\setminus\{\infty\}$ and every $\epsilon\in(0,1]$,
\begin{multline}
 (F,F)_g\\
\le(\#F)\log(\epsilon^{-1})+2(\#F)^2\hat{\eta}_{g,F}(\epsilon)
 +2(\#F)\sup_{\sP^1}|g|
 -2\sum_{w\in F}\log[w,\infty],\label{eq:finite}
\end{multline}
as a consequence of the originals applicable to
$F\subset\bP^1\setminus\{\infty\}$ of
\eqref{eq:negative} and \eqref{eq:lower}.
We note that the term $-2\sum_{w\in F}\log[w,\infty]$ is strictly positive.
Even to obtain the sharper \eqref{eq:mahler} 
in the case $F\subset\bP^1\setminus\{\infty\}$,
we needed \eqref{eq:negative} and \eqref{eq:lower} applicable
to $F$ possibly containing 
$\infty$, and the formula \eqref{eq:formula}.
\end{remark}

\begin{proof}[Proof of Theorem $\ref{th:mahler}$]
 Fix $\epsilon\in(0,1]$. Let us show \eqref{eq:mahler}
 for every non-empty finite subset $F$ in $\bP^1$
 by an induction on $\#F$. First of all, for every singleton $F$ in $\bP^1$, 
 we have $(F,F)_g=0$, so \eqref{eq:mahler} holds.
 Let $N\in\bN$ be $>1$, and suppose that \eqref{eq:mahler} holds
 for every $F\subset\bP^1$ satisfying $\#F=N-1$.

 Fix $F\subset\bP^1$ satisfying $\#F=N$.
 If $\infty\in F$, then $\#(F\setminus\{\infty\})=N-1$.
 By the upper and lower estimates
 \eqref{eq:negative} and \eqref{eq:lower}
 of $((\cZ_F)_\epsilon,(\cZ_F)_\epsilon)_g$, we have
 \begin{align*}
 &2\cdot\sum_{w\in F\setminus\{\infty\}}\Phi_g(w,\infty)
  =2\cdot\sum_{w\in F\setminus\{\infty\}}(\log[w,\infty]-g(w)-g(\infty))\\
 =&2\cdot\sum_{w\in F\setminus\{\infty\}}\log[w,\infty]
 -2\biggl(\sum_{w\in F\setminus\{\infty\}}g(w)+(N-1)\cdot g(\infty)\biggr)\\
 \le&\biggl(-(F,F)_g+2\sum_{w\in F}g(w)
 -N\cdot\log\epsilon+2N^2\cdot\hat{\eta}_{g,F}(\epsilon)\biggr)
\quad(\text{by } 
\eqref{eq:negative}\text{ and }\eqref{eq:lower})\\
&\quad -2\biggl(\sum_{w\in F\setminus\{\infty\}}g(w)+(N-1)\cdot g(\infty)\biggr)\\
 =& -(F,F)_g+N\cdot\log(\epsilon^{-1})
 +2N^2\cdot\hat{\eta}_{g,F}(\epsilon)-2(N-2)g(\infty),
 \end{align*}
 which with the formula \eqref{eq:formula} on $(F,F)_g$ for $z=\infty$ yields
\begin{multline*}
 2(F,F)_g\le N\cdot\log(\epsilon^{-1})
 +2N^2\cdot\hat{\eta}_{g,F}(\epsilon)+2N\cdot\sup_{\sP^1}|g|\\
+(F\setminus\{\infty\},F\setminus\{\infty\})_g,
\end{multline*}
 and we also note that
 $\eta_{g,F\setminus\{\infty\}}(\epsilon)\le\eta_{g,F}(\epsilon)$.
 Hence the induction assumption applied to
 $(F\setminus\{\infty\},F\setminus\{\infty\})_g$
 completes the proof of \eqref{eq:mahler} for the $F$ in this case.
If $\infty\not\in F$, then
there is $h\in U_K$ satisfying $\infty\in h^{-1}(F)$.
Then $\#(h^{-1}(F))=N$. 
By (Lemma \ref{th:pullback} and) Lemma \ref{th:pullbackprod},
we have $(F,F)_g=(h^{-1}(F),h^{-1}(F))_{g\circ h}$, and
by $\sd(h(\cS),h(\cS'))=\sd(\cS,\cS')$ on $\sP^1\times\sP^1$, we also have
$\eta_{g,F}\equiv\eta_{g\circ h,h^{-1}(F)}$ on $[0,1]$.
Hence \eqref{eq:mahler} applied to $(h^{-1}(F),h^{-1}(F))_{g\circ h}$
completes the proof of \eqref{eq:mahler}
for the $F$ in this case.
%
\end{proof}

\begin{proof}[Proof of Theorem $\ref{th:Holder}$]
Suppose that $g|\bP^1$ is a $1/\kappa$-H\"older continuous function
on $(\bP^1,[z,w])$ for some $\kappa\ge 1$ in that 
$C:=\sup_{z,w\in\bP^1:\,\text{distinct}}|g(z)-g(w)|/[z,w]^{1/\kappa}\in\bR_{\ge 0}$.
Recall the definition of $C'\in\bR_{\ge 0}$ in Theorem \ref{th:Holder}.

\begin{lemma}\label{th:comparable}
On $\bP^1\times\sP^1$, 
$|g(z)-g(\cS)|\le C'\cdot\sd(z,\cS)^{1/\kappa}$.
\end{lemma}

\begin{proof}
 If $K$ is archimedean, then there is nothing to show
 since $\sP^1\cong\bP^1$, $\sd(z,w)=[z,w]$, and $C':=C$.
 Suppose that $K$ is non-archimedean. Then $C':=C\cdot 2^{1/\kappa}$, and
 for every $z\in\bP^1$, by the continuity of $g$ and
 $[z,\cdot]_{\can}$ on $\sP^1$ and the density of $\bP^1$ in $\sP^1$,
 we have $|g(z)-g(\cdot)|\le C[z,\cdot]_{\can}^{1/\kappa}$ on $\sP^1$, 
 and in turn
 $|g(z)-g(\cdot)|\le (C\cdot 2^{1/\kappa})\cdot\sd(z,\cdot)^{1/\kappa}$
 on $\sP^1$ by Lemma \ref{th:comparison}.
\end{proof}

 Once Lemma \ref{th:comparable} is at our disposal, for every non-empty
 finite subset $F$ in $\bP^1$, we have 
 $\eta_{g,F}(\epsilon)\le C'\epsilon^{1/\kappa}$ on $[0,1]$, 
 so by setting $\epsilon=(\#F)^{-\kappa}\in(0,1]$
 in \eqref{eq:mahler}, we obtain \eqref{eq:maherHolder}. 
 Now the proof of Theorem \ref{th:Holder} is complete.
\end{proof}

\section{On deduction of \eqref{eq:lower}}\label{sec:lower}

Let $K$ be an algebraically closed field that is complete with respect to a
non-trivial absolute value $|\cdot|$. 
Let $g$ be a continuous weight on $\sP^1$ and, for every
non-empty finite subset $F$ in $\bP^1$, recall the definitions
\eqref{eq:moduluspre} and \eqref{eq:modulus} of the functions
$\eta_{g,F}:[0,1]\to\bR_{\ge 0}$ and $\hat{\eta}_{g,F}:[0,1]\to\bR_{\ge 0}$, 
respectively. 
Let us see that
for every $\epsilon\in(0,1]$ and every $z,w\in\bP^1$,
\begin{multline}
 \int_{\sP^1\times\sP^1}\Phi_g\rd([z]_{\epsilon}\times[w]_{\epsilon})\\
\ge
\begin{cases}
 \Phi_g(z,w)-2\hat{\eta}_{g,\{z,w\}}(\epsilon) & \text{if }z\neq w,\\
 \log\epsilon+2\log[z,\infty]-2g(z)-2\hat{\eta}_{g,\{z\}}(\epsilon)
 & \text{if }z=w\in K,\\
 \log\epsilon-2g(\infty)-2\hat{\eta}_{g,\{\infty\}}(\epsilon)
 & \text{if }z=w=\infty.
 \end{cases}\label{eq:reglower}
\end{multline} 
Once \eqref{eq:reglower} is at our disposal, \eqref{eq:lower}
will follow by a computation similar to that 
in the proof of \cite[Lemma 6.1]{OkuDivisor}.

\begin{remark}
 An estimate similar to \eqref{eq:reglower} was obtained in 
 \cite[Lemma 3.2]{OkuDivisor}, where for archimedean $K$,
 the definition of $[z]_\epsilon$ was slightly different (or, more precisely,
 $[z]_\epsilon$ was defined to be more smooth for archimedean $K$).
\end{remark}

\begin{proof}[Proof of $\eqref{eq:reglower}$]
For every $\epsilon>0$ and every $z,w\in K$, we recall that
 \begin{gather}
 \int_{\sP^1\times\sP^1}\log|\cS-\cS'|_{\infty}\rd([z]_\epsilon\times[w]_\epsilon)(\cS,\cS')
 \ge
 \begin{cases}
  \log|z-w| & \text{if }z\neq w,\\
 \log\epsilon & \text{if }z=w
 \end{cases}\label{eq:regularizationaffine}
 \end{gather}
(see Favre--Rivera-Letelier \cite[Lemme 4.11]{FR06} and 
Fili--Pottmeyer \cite[Lemma 4]{FP15} for non-archimedean
and archimedean $K$, respectively), and
for every $\epsilon\in(0,1]$ and every $z\in K$, we have
 \begin{gather}
 \supp[z]_\epsilon\subset\{\cS\in\sP^1:|\cS-z|_{\infty}\le\epsilon\}
 \subset\{\cS\in\sP^1:\sd(\cS,z)\le\epsilon\}.\label{eq:inclusion}
\end{gather}
By the first inclusion of \eqref{eq:inclusion} and
the density of $\bP^1$ in $\sP^1$, a direct computation shows
that for every $z\in K$ and every $\epsilon\in(0,1]$,
 \begin{multline}
 \sup_{\cS\in\supp[z]_\epsilon}
 \bigl|\log[\cS,\infty]_{\can}-\log[z,\infty]\bigr|
 \le\begin{cases}
    \epsilon & \text{for archimedean }K,\\
    0 & \text{for non-archimedean }K.
   \end{cases}\label{eq:infty}
 \end{multline}

By \eqref{eq:regularizationaffine}, \eqref{eq:inclusion}, and \eqref{eq:infty}
(and the definitions \eqref{eq:kernel} and \eqref{eq:Hsia} of $\Phi_g$
and $|\cS-\cS'|_{\infty}$, respectively), 
we immediately have \eqref{eq:reglower}
for every $z,w\in K$ and every $\epsilon\in(0,1]$, 
and also by $[\iota(\cS),\iota(\cS')]_{\can}=[\cS,\cS']_{\can}$
on $\sP^1\times\sP^1$, for every $\epsilon\in(0,1]$, we have
 \begin{align*}
 &\int_{\sP^1\times\sP^1}\Phi_g\rd([\infty]_\epsilon\times[\infty]_\epsilon)\\
 \ge&\int_{\sP^1\times\sP^1}\log[\cS,\cS']_{\can}\rd([0]_\epsilon\times[0]_\epsilon)(\cS,\cS')-2g(\infty)-2\eta_{g,\{\infty\}}(\epsilon)\\
 \ge&\log\epsilon+2\log[0,\infty]-2g(\infty)-2\hat{\eta}_{g,\{\infty\}}(\epsilon)
 =\log\epsilon-2g(\infty)-2\hat{\eta}_{g,\{\infty\}}(\epsilon).
 \end{align*}

There remains the case where $z=\infty$ and $w\in K$. By 
$\sd(\iota(\cS),\iota(\cS'))=\sd(\cS,\cS')$ on $\sP^1\times\sP^1$
and \eqref{eq:inclusion}, for every $\epsilon\in(0,1]$ and $z\in K$, 
we first have
\begin{multline*}
 \int_{\sP^1\times\sP^1}\Phi_g\rd([\infty]_\epsilon\times[z]_\epsilon)\\
 \ge
 \int_{\sP^1\times\sP^1}
 \log[\cS,\cS']_{\can}\rd([\infty]_\epsilon\times[z]_\epsilon)(\cS,\cS')
 -g(\infty)-g(z)-2\eta_{g,\{\infty,z\}}(\epsilon);
\end{multline*}
moreover, (i) for archimedean $K\cong\bC$, also by $\iota^2=\Id$ on $\sP^1$,
we have
 \begin{align*}
  &\int_{\sP^1\times\sP^1}
  \log[\cS,\cS']_{\can}\rd([\infty]_\epsilon\times[z]_\epsilon)(\cS,\cS')\\
= &\int_{\sP^1\times\sP^1}
 \log[\cS,\cS']_{\can}\rd([0]_\epsilon\times\iota_*[z]_\epsilon)(\cS,\cS')\\
 =&
\int_{\sP^1\times\sP^1}
 \log|\cS-\cS'|_{\infty}\rd([0]_\epsilon\times\iota_*[z]_\epsilon)(\cS,\cS')\\
 &\quad+\int_{\sP^1}\log[\cS,\infty]_{\can}\rd[0]_\epsilon(\cS)
 +\int_{\sP^1}\log[\cS',0]_{\can}\rd[z]_\epsilon(\cS')\\
 \ge&\int_{\sP^1}\log[\cS,\infty]_{\can}\rd[0]_\epsilon(\cS)
 +\int_{\sP^1}\log[\cS',\infty]_{\can}\rd[z]_\epsilon(\cS')\\
 \ge&\log[0,\infty]+\log[z,\infty]-2\epsilon
 =\log[z,\infty]-2\epsilon,
\end{align*}
the former inequality in which holds by
\begin{align*}
& \int_{\sP^1\times\sP^1}
 \log|\cS-\cS'|_{\infty}\rd([0]_\epsilon\times\iota_*[z]_\epsilon)(\cS,\cS')\\
 =&\int_0^{2\pi}\frac{\rd\phi}{2\pi}\int_0^{2\pi}
 \log\left|\epsilon e^{i\theta}-\frac{1}{z+\epsilon e^{i\phi}}\right|\frac{\rd\theta}{2\pi}
 =\int_0^{2\pi}\max\left\{\log\left|\frac{1}{z+\epsilon e^{i\phi}}\right|,\log\epsilon\right\}\frac{\rd\phi}{2\pi}\\
 \ge&-\int_0^{2\pi}\log|(z+\epsilon e^{i\phi})-0|\frac{\rd\phi}{2\pi}
 =-\int_{\sP^1}\log|\cS'-0|_\infty\rd[z]_\epsilon(\cS')\\
=&-\int_{\sP^1}\log[\cS',0]_{\can}\rd[z]_\epsilon(\cS')
+\int_{\sP^1}\log[\cS',\infty]_{\can}\rd[z]_\epsilon(\cS'),
\end{align*}
and (ii) for non-archimedean $K$, 
also by the definition \eqref{eq:Gromov} of $[\cS,\cS']_{\can}$,
we have
\begin{align*}
&\int_{\sP^1\times\sP^1}
\log[\cS,\cS']_{\can}\rd([\infty]_\epsilon\times[z]_\epsilon)(\cS,\cS')\\
=&\log[\iota(\pi_{\epsilon}(0)),\pi_\epsilon(z)]_{\can}
\biggl(=
\begin{cases}
 \log\diam\cS_{\can}=1 & \text{if }|z|\le 1,\\
 \log[\pi_\epsilon(0),\iota(\pi_\epsilon(z))]_{\can}\ge\log[0,\iota(z)]
& \text{otherwise}
\end{cases}\biggr)\\
\ge&\log[\infty,z].
\end{align*}

Hence we have
$\int_{\sP^1\times\sP^1}\Phi_g\rd([\infty]_\epsilon\times[z]_\epsilon)
\ge\Phi_g(\infty,z)-2\hat{\eta}_{g,\{\infty,z\}}(\epsilon)$
for every $\epsilon\in(0,1]$ and every $z\in K$, and
the proof of \eqref{eq:reglower} is complete.
\end{proof}

\section{On the optimality of \eqref{eq:asymp}}\label{eq:Lipschitz}
Let $K$ be an algebraically closed field that is complete with respect to a
non-trivial absolute value $|\cdot|$. We include a few examples of 
normalized weights $g$ on $\sP^1$, for which \eqref{eq:asymp} is optimal.

For archimedean $K$, the function
$z\mapsto\log\max\{1,|z|\}+\log[z,\infty]$
on $K\cong\bC$ extends to
a normalized weight $g_0$ on $\sP^1\cong\bP^1$, which is
Lipschitz, i.e.,
$1/\kappa$-H\"older continuous on $(\bP^1,[z,w])$ for $\kappa=1$,
by the piecewise smoothness of $g_0$ on $\bP^1$. 
Moreover, for every $N\in\bN$,
$g_0\equiv\log[\cdot,\infty]$ on $F_N:=\{e^{2i\pi k/N}:k\in\{0,1,\ldots,N-1\}\}$
and
\begin{multline*}
\frac{\sum_{z\in F_N}\sum_{w\in F_N\setminus\{z\}}\Phi_{g_0}(z,w)}{(\#F_N)\log(\#F_N)}\\
=
 \frac{\sum_{k=0}^{N-1}\sum_{j\in\{0,1,\ldots,N-1\}\setminus\{k\}}\log|e^{2i\pi k/N}
 -e^{2i\pi j/N}|}{N\log N}\\
 =\frac{\sum_{k=0}^{N-1}\log\left|(z^N-1)'|_{z=e^{2i\pi k/N}}\right|}{N\log N}
 =\frac{N\log|N|}{N\log N}=1,
\end{multline*}
which implies that \eqref{eq:asymp} is optimal for $g_0$.

For non-archimedean $K$, we can fix $d\in\bN$ such that
$d>1$ and that $|d|=1$
(note that if $k\in\bN$ satisfies $|k|<1$, 
then $|k+1|=\max\{|k|,|1|\}=1$).
Fix $\lambda\in K$ such that $|\lambda|>d^{d/(d-1)}(>1)$ 
and set $f_\lambda(z):=z^d+\lambda\in K[z]$. Then we have
$g_{f_\lambda}(\cS)
=-\lim_{n\to\infty}(\log[f_\lambda^n(\cS),\infty]_{\can})/d^n+\log[\cS,\infty]_{\can}$ on $\sP^1$. 

For every distinct $z,w\in K$, 
$(f_\lambda(z)-f_\lambda(w))/(z-w)=\sum_{j=0}^{d-1}z^jw^{d-1-j}$, so that
by the strong triangle inequality, we have
\begin{gather}
\frac{|f_\lambda(z)-f_\lambda(w)|}{|z-w|}\le\max\{|z|,|w|\}^{d-1}.\label{eq:average}
\end{gather}
%
%
By $|\lambda|>1$ (and the strong triangle inequality), 
we also have
$f_\lambda(\{z\in K:|z|\le|\lambda|^{1/d}\})\subset\{z\in K:|z|\le|\lambda|\}$,
$f_\lambda(\{z\in K:|z|<|\lambda|^{1/d}\})\subset\{z\in K:|z|=|\lambda|\}
\subset\{z\in K:|z|>|\lambda|^{1/d}\}$, and 
on $\{z\in K:|z|>|\lambda|^{1/d}\}$,
$f_\lambda(z)=z^d$
(so $|f_\lambda(z)|=|z|^d>|z|>|\lambda|^{1/d}>1$).
In particular, 
\begin{gather}
 \sup_{z\in K,n\in\bN}\frac{\max\{1,|z|\}}{\max\{1,|f_\lambda^n(z)|\}}
 \le|\lambda|^{1/d},\label{eq:image}
\end{gather}
and 
for every $n\in\bN$, setting
$P_n:=\{z\in K:f_\lambda^n(z)=z\}$,  
we have $P_n\subset\{z\in K:|z|=|\lambda|^{1/d}\}$,
$\#P_n=d^n$, and $g_{f_\lambda}\equiv\log[\cdot,\infty]$ on $P_n$, 
and also by the chain rule and $|d|=1$, 
$|(f_\lambda^n)'|\equiv(|d||\lambda|^{(d-1)/d})^n=(|\lambda|^{(d-1)/d})^n(>1)$
on $P_n$. 

Hence, for every $n\in\bN$, using also the strong triangle inequality, we have
\begin{multline*}
 \frac{\sum_{z\in P_n}\sum_{w\in P_n\setminus\{z\}}\Phi_{g_{f_\lambda}}(z,w)}{(\#P_n)\log(\#P_n)}
 =\frac{\sum_{w\in P_n}\log\left|(f_\lambda^n(z)-z)'|_{z=w}\right|}{d^n\cdot\log(d^n)}\\
 = \frac{\log((|\lambda|^{(d-1)/d})^n)}{n\log d}
=\frac{\log(|\lambda|^{(d-1)/d})}{\log d}.
\end{multline*}
On the other hand, for every $n\in\bN$ and every distinct $z,w\in K$, 
setting 
\begin{gather*}
 t:=\min\bigl\{j\in\{0,1,\ldots,n-1\}:
 \max\{|f^j(z)|,|f^j(w)|\}>|\lambda|^{1/d}\bigr\}
\end{gather*}
(under the convention that $\min\emptyset=-\infty$) 
and $L_t:=\max\{|f^t(z)|,|f^t(w)|\}$ unless $t=-\infty$, 
recalling the definition of
$[z,w]$ in Notation \ref{th:notation} and using \eqref{eq:average} and
\eqref{eq:image}, we have 
\begin{align*}
&\frac{[f_\lambda^n(z),f_\lambda^n(w)]}{[z,w]}\\
\le&
\biggl(\prod_{j=0}^{n-1}\max\{|f_\lambda^j(z)|,|f_\lambda^j(w)|\}\biggr)^{d-1}
\cdot\frac{\max\{1,|z|\}}
{\max\{1,|f_\lambda^n(z)|\}}
\cdot\frac{\max\{1,|w|\}}{\max\{1,|f_\lambda^n(w)|\}}\\
\le&
\begin{cases}
((|\lambda|^{1/d})^n)^{(d-1)}\cdot(|\lambda|^{1/d})^2
& \text{if }t=-\infty,\\
\frac{\bigl(\max\{|z|,|w|\}^{(d^n-1)/(d-1)}\bigr)^{(d-1)}\cdot\max\{|z|,|w|\}}{\max\{|z|,|w|\}^{d^n}}\cdot|\lambda|^{1/d}
& \text{if }t=0,\\
((|\lambda|^{1/d})^{t-1})^{d-1}
\cdot\frac{\bigl(L_t^{(d^{n-t}-1)/(d-1)}\bigr)^{(d-1)}\cdot L_t}{L_t^{d^{n-t}}}\cdot|\lambda|^{1/d}
& \text{if }t\in\{1,\ldots,n-1\}
\end{cases}\\
\le& ((|\lambda|^{1/d})^n)^{(d-1)}\cdot(|\lambda|^{1/d})^2,
\end{align*}
so that for every $n\in\bN$, using also $|\lambda|>d^{d/(d-1)}$,
we have
 \begin{gather*}
 \sup_{z,w\in\bP^1:\,\text{distinct}}\frac{[f_\lambda^n(z),f_\lambda^n(w)]}{[z,w]}\le
 ((|\lambda|^{1/d})^n)^{d-1}\cdot(|\lambda|^{1/d})^2>d^n.
 \end{gather*}
In particular, by Fact $\ref{th:dynamics}$,
$g_{f_\lambda}|\bP^1$ is $1/\kappa$-H\"older continuous on 
$(\bP^1,[z,w])$ for every $\kappa>(\log(|\lambda|^{(d-1)/d}))/\log d$.
Hence \eqref{eq:asymp} is optimal for $g_{f_\lambda}$.


\begin{acknowledgement}
The author thanks the referee for a very careful scrutiny and
comments, and Professor Matt Baker for invaluable discussions.
This research was partially supported by JSPS Grant-in-Aid 
for Scientific Research (C), 15K04924.
\end{acknowledgement}

\def\cprime{$'$}

\end{document}